\documentclass[a4paper,10pt]{amsart}

\usepackage{a4}
\usepackage{graphics,graphicx}
\usepackage{amssymb,amsmath}
\usepackage[all]{xy}
\usepackage{stmaryrd}
\CompileMatrices
\usepackage[breaklinks]{hyperref}
\usepackage{tikz}
\usetikzlibrary{decorations.markings}
\usetikzlibrary{decorations.pathreplacing}
\usepackage{enumitem}
\usepackage{todonotes}

\newcommand{\I}{\mathrm{I}}

\newtheorem{theorem}{Theorem}[section]
\newtheorem{proposition}[theorem]{Proposition}
\theoremstyle{definition}

\newtheorem{example}[theorem]{Example}
\newtheorem{remark}[theorem]{Remark}

\theoremstyle{remark}
\newtheorem*{acknowledgement}{Acknowledgement}

\def\tt{\ensuremath{\mathfrak{t}}}

\def\KK{\mathbb{K}}
\def\ZZ{\mathbb{Z}}

\def\PP{\mathbb{P}}
\def\QQ{\mathbb{Q}}

\def\CC{\mathcal C}

\def\<{\langle}
\def\>{\rangle}

\def\cox{\mathcal{R}}

\renewcommand{\phi}{\varphi}

\def\rq#1{\widehat{#1}}

\def\b#1{\overline{#1}}
\def\bangle#1{\langle #1 \rangle}

\def\CC{{\mathbb C}}
\def\KK{{\mathbb K}}

\def\ZZ{{\mathbb Z}}

\def\QQ{{\mathbb Q}}
\def\PP{{\mathbb P}}

\def\Cox{\cox}

\def\Cl{\operatorname{Cl}}

\def\Pic{\operatorname{Pic}}

\def\KT#1{\KK[T_1,\ldots,T_{#1}]}

\def\tt#1{\texttt{#1}}
\def\out#1{\begingroup\tiny\begin{gather*} #1 \end{gather*}\endgroup}

\begin{document}

\title[MDSpackage]{A software package for Mori dream spaces} 
\author[J.~Hausen and S.~Keicher]{J\"urgen~Hausen and Simon~Keicher}
\subjclass[2000]{
14Q10, 14Q15\\
The second author was partially supported 
by the DFG Priority Program SPP 1489.
}

 \address{Mathematisches Institut, Universit\"at T\"ubingen,
Auf der Morgenstelle 10, 72076 T\"ubingen, Germany}
\email{juergen.hausen@uni-tuebingen.de}

\address{Mathematisches Institut, Universit\"at T\"ubingen,
Auf der Morgenstelle 10, 72076 T\"ubingen, Germany}
\email{keicher@mail.mathematik.uni-tuebingen.de}

\begin{abstract}
Mori dream spaces form a large example class of 
algebraic varieties, comprising the well known
toric varieties.
We provide a first software package for the explicit 
treatment of Mori dream spaces and demonstrate its 
use by presenting basic sample computations.
The software package is accompanied by a Cox ring 
database which delivers defining data for Cox rings 
and Mori dream spaces in a suitable format.
As an application of the package, we determine 
the common Cox ring for the symplectic resolutions 
of a certain quotient singularity investigated
by Bellamy/Schedler and Donten-Bury/Wi\'sniewski.
\end{abstract}

\maketitle


\section{Introduction}

By a \emph{Mori dream space} we mean here a normal 
complete variety $X$ defined over an algebraically 
closed field $\KK$ of characteristic zero having 
a finitely generated divisor class group $\Cl(X)$ 
and a finitely generated \emph{Cox ring}
$$ 
\mathcal{R}(X)
\ = \ 
\bigoplus_{\Cl(X)} \Gamma(X,\mathcal{O}(D)),
$$  
where we refer to~\cite{ArDeHaLa} for the details 
of the definition.
Mori dream spaces have been introduced by Hu and 
Keel~\cite{HuKe} as a class of varieties with an 
optimal behaviour with respect to the Minimal 
Model Program.
Well known examples of Mori dream spaces are 
toric and, more generally, spherical varieties,
smooth Fano varieties~\cite{BCHM} or Calabi-Yau
varieties with a polyhedral effective 
cone~\cite{McK}.
Examples of general type can be obtained by 
Lefschetz-type theorems~\cite{ArLa,Jow}.

An important feature of Mori dream spaces is
that they allow an explicit encoding in terms 
of algebraic and combinatorial 
data~\cite{ArDeHaLa, BeHa,Ha2} and their theory
has close relations to toric geometry.
This turns Mori dream spaces into natural candidates 
for extending the ``testing ground of algebraic 
geometry'' given by the toric varieties.
For working explicitly with toric varieties, 
there are meanwhile several software packages 
available such as~\cite{tordiv14, magma, Gu, Joytoric, KLPW, VeJo}.
Our intention is to provide with 
\texttt{MDSpackage}~\cite{mdspackage} such a tool 
also for the larger class of Mori dream spaces.

Let us give a more concrete impression.
Every Mori dream space $X$ is encoded by its Cox ring $\Cox(X)$ 
plus data located in the divisor class group $\Cl(X)$ 
which fix the isomorphy type among all varieties sharing 
$\Cox(X)$ as Cox ring.
For instance, we can define a three-dimensional projective 
Mori dream space $X$ by prescribing its divisor class group 
as $K := \ZZ^2$, its $K$-graded Cox ring as 
\begin{center}
\vspace*{-2ex}
\begin{minipage}{7cm}
\begin{gather*}
R
 \, :=\, 
\KT{6}/\<T_1T_2 + T_3T_4 + T_5^2 + T_6^2\>,
\\
  Q\, :=\, 
 \left[
\mbox{\tiny $
\begin{array}{rrrrrr}
-2 & 2 & -1 & 1 & 0 & 0\\
 1 & 1 &  1 & 1 & 1 & 1
\end{array}
$}
\right],
\end{gather*}
\end{minipage}
\qquad\quad
\begin{minipage}{3cm}
\begin{tikzpicture}[scale=.73, transform shape]
\fill[color=black!20] (-2,1) -- ( 2,1) -- (0,0) -- cycle;
\fill[color=blue!40!white] (-1,1) -- (1,1) -- (0,0) -- cycle;

\draw[thick] (0,0) -- (0,1) node[anchor=south]{$q_6=q_5$};
\draw[thick] (0,0) -- (-1,1) node[anchor=south]{$q_3$};
\draw[thick] (0,0) -- (-2,1) node[anchor=south]{$q_1$};
\draw[thick] (0,0) -- (1,1) node[anchor=south]{$q_4$};
\draw[thick] (0,0) -- (2,1) node[anchor=south]{$q_2$};

\foreach \a/\b in {0/1,-1/1,-2/1, 1/1, 2/1} {
  \fill (\a,\b) circle(.2em);
}

\draw[color=blue!80!white,decorate,decoration={brace,amplitude=7pt},rotate=0] (-1,1.5) -- (1,1.5);
\draw[color=blue!80!white] (0,1.75) node[anchor=south]{${\rm Mov}(X)$};
\end{tikzpicture}
\end{minipage}
\end{center}
where $Q$ has the generator degrees $\deg(T_i) \in K$ 
as its columns,
and an ample class $w \in K$ taken from the relative 
interior of the above blue cone, i.e.~the prospective moving 
cone of~$X$.
The $K$-grading of $R$ defined by $Q$ gives rise to 
an action of the 2-torus $H = (\KK^*)^2$ on 
$$
\b{X} 
\ = \ 
V(T_1T_2 + T_3T_4 + T_5^2 + T_6^2)
\ \subseteq \ 
\KK^6
$$ 
and the Mori dream space $X$ is the quotient by $H$ of the set
$\rq{X} \subseteq \b{X}$ of semistable points 
associated to the weight $w$.
As mentioned, divisor class group and Cox ring of $X$ 
are given by $\Cl(X) = K$ and $\Cox(X) = R$.
The task of \texttt{MDSpackage} is then to extract further
geometric invariants and properties from these defining data.
For example, it determines the Picard group 
and the singularities of $X$ as
\begin{gather*}
\Pic(X)
\ =\ 
6\ZZ \oplus 3\ZZ \ \subseteq \ \ZZ^2
\ =\ 
\Cl(X),
\\
{\rm Sing}(X) 
\ =\  
\{
\{1, 5, 6\}, \{1, 2, 5, 6\}, \{1, 2, 6\}, \{2, 3\}, \{1, 4\}, \{1, 2, 5\}
\},
\end{gather*}
where the output on the singularities provides information
on their Cox coordinates; for example 
$\{1, 5, 6\}$ says that there is a singular point $x \in X$ 
stemming from a point $z \in \rq{X} \subseteq \KK^6$ having
precisely $z_1,z_5,z_6$ as non-zero coordinates.

\texttt{MDSpackage} aims to be an easy-to-use 
computing environment for up to medium size computations.
In the subsequent section,
we present sample computations and thereby explain
the syntax.
The major computational ingredients come from commutative 
algebra and polyhedral combinatorics. 
The basic features of \texttt{MDSpackage} are
\begin{itemize}
\item basics on finitely
generated abelian groups
and algebras graded by them,
\item 
computing Picard group, local class groups,
cones of effective,
movable or semiample divisor classes,
Mori chamber decomposition, 
pullback of $\QQ$-Cartier divisors, 
\item 
computing the
canonical toric ambient variety, 
induced orbit stratification,
irrelevant ideal,
\item
testing (quasi-)smoothness,
($\QQ$-)factoriality, 
completeness, (quasi-)projecti\-vity,
\item
computing the singular locus, global resolution of singularities 
(approved for varieties with torus action of complexity 
one, experimental in the general case),
\item
for complete intersection Cox rings:
computing intersection numbers,
graph of exceptional divisors,
anticanonical divisor class,
Gorenstein index,
testing ($\QQ$-)Gorenstein and Fano properties,
\item 
for varieties with a torus action of complexity one:
tests for being ($\varepsilon$-log) terminal,
almost homo\-gen\-eous,
computing roots of the automorphism group.
\end{itemize}
Detailed background on the algorithms and a complete 
manual for~\texttt{MDSpackage} can be found in~\cite{Ke}.
Moreover, a comprehensive online manual is available 
at~\cite{mdspackage}.
Our package is implemented in the computer algebra system 
\texttt{Maple}~\cite{maple} and makes use of the \texttt{convex}
package by Matthias Franz~\cite{convex}.

\texttt{MDSpackage} has been essentially used in the 
classification of Fano varieties. In~\cite{BeHaHuNu}
the $\QQ$-factorial terminal Fano threefolds of Picard
number one with an effective action of a two-dimensional
torus are classified; among other computations performed 
with \texttt{MDSpackage}, this involves more than $10^{6}$
terminality tests for possible candidates. 
Other applications are the classification results on 
$\KK^{*}$-surfaces of high Gorenstein index given 
in~\cite{Ke}.

In order to open a broad computer-supported access to
Mori dream spaces, we are building up a database 
of Cox rings~\cite{coxringdb}, the entries of which can 
be exported in a suitable data format for \texttt{MDSpackage}.
Our motivation for creating such a combined toolkit is the fruitful
linking of the theory of weighted complete intersections 
with the \emph{graded ring database}, 
see~\cite{gradedringdb, IaFl}.

In Section~\ref{sec:app} we apply~\texttt{MDSpackage} to 
continue work of Bellamy/Schedler~\cite{BellamySchedler} 
and Donten-Bury/Wi\'{s}niewski~\cite{JarekMaria} on 
the 81 symplectic resolutions of a certain quotient 
singularity: in Theorem~\ref{thm:81resol} we determine 
the common Cox ring of these resolutions.

\section{Working with~\texttt{MDSpackage}}
\label{sec:working}

According to~\cite{BeHa,Ha2}, a Mori dream space $X$ is 
encoded by a \emph{bunched ring}.
This basically is an integral algebra $R = \oplus_K R_{w}$ 
graded by a finitely generated abelian group $K$ such that 
the $K$-homogeneous elements admit unique factorization
together with a collection $\Phi$, called ``bunch'', of 
convex polyhedral cones in $K \otimes \QQ$. 
Concretely, the $K$-graded algebra $R$ is given by 
generators and relations and a degree map assigning to 
each generator its $K$-degree:
$$
R
\ = \ 
\KK[T_1,\ldots,T_r]
\, / \, 
\bangle{g_1,\ldots, g_s},
\qquad\qquad
Q \colon \ZZ^r \ \to \ K,
\qquad
e_i \ \mapsto \ \deg(T_i).
$$
The bunch of cones $\Phi$ fixes the isomorphy type of 
$X$ among all varieties having $R$ as Cox ring. 
In case of a projective Mori dream space $X$, one can 
simply define $\Phi$ by fixing an ample class $w \in K$ 
from inside the moving cone.

We now demonstrate the practical work 
with~\texttt{MDSpackage}~\cite{mdspackage}
by means of three example computations, 
more can be found in~\cite{Ke}.

\begin{example}
The first step in defining a Mori dream space~$X$ 
with \texttt{MDSpackage} is to enter the abelian groups 
$E := \ZZ^r$, $K$ and the degree map $Q \colon E \to K$:\\

\begingroup
\footnotesize
\begin{enumerate}[leftmargin=3em]
\item[\tt >] \tt{E := createAG(8);}
\out{E := AG(8, [])}
\item[\tt >] \tt{K := createAG(3, [2]);}
\out{K := AG(3, [2])}
\item[\tt >] \tt{A := cols2matrix([[1,0,1,1], [1,1,1,0], [0,1,1,1], [0,-1,1,0],
 [-1,-1,1,1], \break [-1,0,1,0], [2,1,1,1], [-2,-1,1,0]]);}
 \out{A := \left[\begin{array}{rrrrrrrr}
           1 & 1 & 0 & 0 & -1 & -1 & 2 & -2\\
	   0 & 1 & 1 & -1 & -1 & 0 & 1 & -1\\
	   1 & 1 & 1 & 1 & 1 & 1 & 1 & 1\\
	   1 & 0 & 1 & 0 & 1 & 0 & 1 & 0
           \end{array}\right]}
\item[\tt >] \tt{Q := createAGH(E, K, A);}
\out{Q := AGH([8, []], [3, [2]])}
\end{enumerate}
\endgroup

\noindent
Here we took $E=\ZZ^8$ and $K=\ZZ^3\oplus \ZZ/2\ZZ$
and the $4\times 8$ matrix $A$ fixes a map  $\ZZ^8\to \ZZ^4$
inducing the degree map $Q \colon E \to K$.
The next step is to define the $K$-graded ring $R$.
We have to specify variables, relations and the grading:\\

\begingroup
\footnotesize
\begin{enumerate}[leftmargin=3em]
\item[\tt >] \tt{TT := vars(8);}
\out{TT := \left[
T[1], T[2], T[3], T[4], T[5], T[6], T[7], T[8]
\right]}
\item[\tt >] \tt{RL := [T[1]*T[6] + T[2]*T[5] + T[3]*T[4] + T[7]*T[8]];}
\out{RL := \left[T[1]T[6] + T[2]T[5] + T[3]T[4] + T[7]T[8]\right]}
\item[\tt >] \tt{R := createGR(RL, TT, [Q], 'nocheck');}
\out{R := GR(8, 1, [3, [2]])}
\end{enumerate}
\endgroup

\noindent
The last output line indicates that $R$ is given by eight 
generators and one relation and its grading group is 
isomorphic to $\ZZ^3\oplus \ZZ/2\ZZ$; the option \tt{'nocheck'}
speeds up the computation by omitting plausibility checks.
To define $X$, 
it remains to fix a prospective 
ample class $w \in K\otimes \QQ$:\\

\begingroup
\footnotesize
\begin{enumerate}[leftmargin=3em]
\item[\tt >] \tt{w := [0,0,2];}
\out{w := [0,0,2]}
\item[\tt >] \tt{X := createMDS(R, w);}
\out{X := MDS(8, 1, 4, [3, [2]])}
\end{enumerate}
\endgroup

\noindent
The last output line is similar to that for $R$,
the only new thing is the third entry, saying 
that the resulting Mori dream space $X$ is of dimension four. 
We are now ready for computations with~$X$.
First, we determine the Picard group of $X$ 
as a subgroup of $\Cl(X)=K$
and the factor group $\Cl(X) / \Pic(X)$:\\

\begingroup
\footnotesize
\begin{enumerate}[leftmargin=3em]
\item[\tt >] \tt{Pic := MDSpic(X);}
\out{Pic := AG(3, [])}
\item[\tt >] \tt{AGfactgrp(K, Pic);}
\out{AG(0, [2, 12, 12, 24])}
\end{enumerate}
\endgroup

\noindent
The first output tells us tat $\Pic(X)$ is of rank three and 
torsion free, i.e., we have $\Pic(X) \cong \ZZ^3$.
The second output says
\[
\Cl(X) / \Pic(X)
\ \cong \ 
\ZZ/2\ZZ
\oplus
\ZZ/12\ZZ
\oplus
\ZZ/12\ZZ
\oplus
\ZZ/24\ZZ.
\]
The cones of semiample, movable and effective
divisor classes in the rational divisor class group 
$K \otimes_\ZZ \QQ$ are computed as follows:\\

\begingroup
\footnotesize
\begin{enumerate}[leftmargin=3em]
\item[\tt >] \tt{MDSsample(X); }
\out{CONE(3, 3, 0, 8, 8)}
\item[\tt >] \tt{MDSmov(X); }
\out{CONE(3, 3, 0, 4, 4)}
\item[\tt >] \tt{MDSeff(X);}
\out{CONE(3, 3, 0, 4, 4)}
\end{enumerate}
\endgroup

\noindent
These cones are stored in \texttt{convex} format.
In particular, for the first one, the output tells 
us that we have a $3$-dimensional cone in $3$-space 
having $0$-dimensional lineality part and $8$ rays and $8$ 
facets; 
further information on the cones can be extracted 
via suitable \texttt{convex} commands.
We compute the Mori chamber decomposition of the 
effective cone:\\

\begingroup
\footnotesize
\begin{enumerate}[leftmargin=3em]
\item[\tt >] \tt{F := MDSchambers(X);}
\out{F := FAN(3, 0, [0, 0, 37])}
\end{enumerate}
\endgroup

\noindent
Again the fan is stored in \texttt{convex} format.
We refer to~\cite{HuKe} for theoretical background 
of the Mori chamber decomposition and to~\cite{Ke0} 
for the algorithmic aspects.
We can also ask \texttt{MDSpackage} for 
a \texttt{povray}-visualization:\\

\begingroup
\footnotesize
\begin{enumerate}[leftmargin=3em]
\item[\tt >] \tt{render(F);}
\end{enumerate}
\endgroup

\begin{center}
\includegraphics[width=5cm]{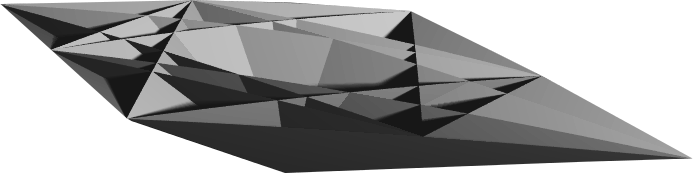}
\end{center}

\noindent
For complete intersection Cox rings like the current example, 
we may also check if~$X$ is a Fano variety and compute its 
Gorenstein index:\\

\begingroup
\footnotesize
\begin{enumerate}[leftmargin=3em]
\item[\tt >] \tt{MDSisfano(X);}
\out{true}
\item[\tt >] \tt{MDSgorensteinind(X);}
\out{4}
\end{enumerate}
\endgroup
\end{example}

\begin{example}
We consider the Gorenstein del Pezzo surface $X$ 
having Picard number one, singularity type $E_6A_2$
and finite automorphism group. 
The Cox ring is given by 
\[
\Cox(X) 
\,=\, \KT{4}/\<-T_{1}T_{4}^{2}+T_{2}^{3}+T_{2}T_{3}T_{4}+T_{3}^{3}\>
,\quad
Q \,=\, 
\left[
    \mbox{\tiny $
    \begin{array}{rrrr}
    1 & 1 & 1 & 1 \\
    \b 1 & \b 2 & \b 0 & \b 1
    \end{array}
    $}
    \right],
\]
where the divisor class group $\Cl(X)$ is isomorphic 
to $K = \ZZ \oplus \ZZ/3\ZZ$, see~\cite[Thm.~4.1]{HaKeLa}. 
First, we have to enter $X$:\\

\begingroup
\footnotesize
\begin{enumerate}[leftmargin=3em]
\item[\tt >] \tt{TT := vars(4);}
\out{TT := [T[1],T[2],T[3],T[4]]}
\item[\tt >] \tt{RL := [-T[1]*T[4]\textasciicircum 2 + T[2]\textasciicircum 3 + T[2]*T[3]*T[4] + T[3]\textasciicircum 3];}
\out{RL := [-T[1]T[4]^2 + T[2]^3 + T[2]T[3]T[4] + T[3]^3]}
\item[\tt >] \tt{B := cols2matrix([[1,1], [1,2], [1,0], [1,1]]);}
\out{B := \left[
\begin{array}{rrrr}
 1 & 1 & 1 & 1 \\
 1 &  2 &  0 &  1
\end{array}
\right]}
\item[\tt >] \tt{Q := createAGH(createAG(4), createAG(1, [3]), B);}
\out{Q := AGH([4, []], [1, [3]])}
\item[\tt >] \tt{R := createGR(RL, TT, [Q]);}
\out{R := GR(4, 1, [1, [3]])}
\item[\tt >] \tt{X := createMDS(R, [1]);}
\out{X := MDS(4, 1, 2, [1, [3]])}
\end{enumerate}
\endgroup

\noindent
We take a look at the singularities of $X$ and compute
a resolution.
Let us check smoothness properties and determine the singular locus:\\

\begingroup
\footnotesize
\begin{enumerate}[leftmargin=3em]
\item[\tt >] \tt{MDSissmooth(X);}
\out{false}
\item[\tt >] \tt{MDSisquasismooth(X);}
\out{false}
\item[\tt >] \tt{L := MDSsing(X);}
\out{L := \Bigl[\bigl[[-2T[1]T[4]+T[2]T[3], T[2]T[4]+3T[3]^2, 3T[2]^2+T[3]T[4], -T[4]^2,\\
-T[1]T[4]^2+T[2]^3+T[2]T[3]T[4]+T[3]^3],\ 
[T[1],T[2],T[3],T[4]]\bigr],\ 
[\{4\}, \{1\}]\Bigr]}
\end{enumerate}
\endgroup

\noindent
The first entry of \tt{L} is a list of generators for the ideal
of the singular locus of the total coordinate space $\b{X}$ 
and a list of used variables.
As mentioned, the second entry tells us that the singularities 
are the points $[0,0,0,1]$ and $[1,0,0,0]$, given in Cox coordinates.
The minimal resolution $X_2 \to X$ is computed by\\

\begingroup
\footnotesize
\begin{enumerate}[leftmargin=3em]
\item[\tt >] \tt{X2 := MDSresolvesing(X, 'verify', 'minimal');}\\
{\footnotesize Verification successful.}
\out{X2 := MDS(12, 1, 2, [9, []])}
\end{enumerate}
\endgroup

\noindent
Since the verification was successful, 
this means that $X_2$ is indeed a smooth Mori dream space.
In accordance with~\cite[p.~40, type $E_6A_2$]{Der},
we obtain for the Cox ring of $X_2$ a single defining equation
and the following degree matrix:\\

\begingroup
\footnotesize
\begin{enumerate}[leftmargin=3em]
\item[\tt >] \tt{R2 := MDSdata(X2)[1];}
\out{R2 := GR(12, 1, [9, []])}
\item[\tt >] \tt{GRdata(R2)[1];}
\out{
\bigl[
-T[1]T[4]^2T[5]+T[2]^3T[12]T[7]^2T[8]+T[2]T[3]T[11]T[4]T[5]T[6]T[7]T[8]T[10]\\
+T[3]^3T[9]T[11]^2T[10]
\bigr]}
\item[\tt >] \tt{GRdata(R2)[3];}
\out{
\left[
AGH([12, []], [9, []]),\ 
\left[
    \mbox{\tiny $
    \begin{array}{rrrrrrrrrrrr}
    1 & 1 & 1 & 1 & 0 & 0 & 0 & 0 & 0 & 0 & 0 & 0 \\
    1 & 0 & 0 & -1 & 1 & 0 & 0 & 0 & 0 & 0 & 0 & 0 \\
    1 & 0 & 0 & 0 & -1 & 1 & 0 & 0 & 0 & 0 & 0 & 0 \\
    1 & 0 & 0 & 0 & 0 & 0 & 0 & 1 & 1 & 0 & 0 & 0 \\
    0 & -1 & 0 & 0 & -1 & 0 & 1 & 0 & -1 & 0 & 0 & 0 \\
    0 & 0 & 0 & 0 & 0 & -1 & 0 & 0 & -1 & 1 & 0 & 0 \\
    1 & 0 & 0 & 0 & 0 & 0 & 1 & -1 & 0 & 1 & 0 & 0 \\
    0 & 1 & 0 & 1 & 0 & 0 & 0 & -1 & 0 & 0 & 1 & 0 \\
    0 & 0 & 0 & 0 & 1 & 0 & 0 & 0 & 1 & 0 & 0 & 1
    \end{array}
    $}
    \right]
\right]
}
\end{enumerate}
\endgroup

\noindent
We compute the graph of exceptional curves and
a list of the self-intersection number of the invariant divisors
$V(X; T_i)$; looking at the subgraph of $(-2)$-curves, we 
then see the $E_6$- and the $A_2$-singularities:\\

\begingroup
\footnotesize
\begin{enumerate}[leftmargin=3em]
\item[\tt >] \tt{MDSintersgraph(X2, 'latex');}
\begin{center}

\begin{tikzpicture}[scale=.75, transform shape]
    \tikzstyle{vertex}=[circle,fill=black!30,minimum size=15pt,inner sep=0pt]
    \tikzstyle{vertexblack}=[circle,fill=black,minimum size=15pt,inner sep=0pt]
    \tikzstyle{vertexwhite}=[circle,fill=white,minimum size=15pt,draw=black,inner sep=0pt]
    \pgfmathsetmacro{\angle}{(360 / 12) * 1};
    	\node[vertex,xshift=6cm,yshift=.5cm] (T1) at (\angle:1.5cm) {$T_{1}$};
    \pgfmathsetmacro{\angle}{(360 / 12) * 2};
    	\node[vertex,xshift=6cm,yshift=.5cm] (T6) at (\angle:1.5cm) {$T_{6}$};
    \pgfmathsetmacro{\angle}{(360 / 12) * 3};
    	\node[vertex,xshift=6cm,yshift=.5cm] (T9) at (\angle:1.5cm) {$T_{9}$};
    \pgfmathsetmacro{\angle}{(360 / 12) * 4};
    	\node[vertex,xshift=6cm,yshift=.5cm] (T12) at (\angle:1.5cm) {$T _{12}$};
    \pgfmathsetmacro{\angle}{(360 / 12) * 5};
    	\node[vertex,xshift=6cm,yshift=.5cm] (T2) at (\angle:1.5cm) {$T_{2}$};
    \pgfmathsetmacro{\angle}{(360 / 12) * 6};
    	\node[vertex,xshift=6cm,yshift=.5cm] (T3) at (\angle:1.5cm) {$T_{3}$};
    \pgfmathsetmacro{\angle}{(360 / 12) * 7};
    	\node[vertex,xshift=6cm,yshift=.5cm] (T4) at (\angle:1.5cm) {$T_{4}$};
    \pgfmathsetmacro{\angle}{(360 / 12) * 8};
    	\node[vertex,xshift=6cm,yshift=.5cm] (T5) at (\angle:1.5cm) {$T_{5}$};
    \pgfmathsetmacro{\angle}{(360 / 12) * 9};
    	\node[vertex,xshift=6cm,yshift=.5cm] (T7) at (\angle:1.5cm) {$T_{7}$};
    \pgfmathsetmacro{\angle}{(360 / 12) * 10};
    	\node[vertex,xshift=6cm,yshift=.5cm] (T8) at (\angle:1.5cm) {$T_{8}$};
    \pgfmathsetmacro{\angle}{(360 / 12) * 11};
    	\node[vertex,xshift=6cm,yshift=.5cm] (T10) at (\angle:1.5cm) {$T_{10}$};
    \pgfmathsetmacro{\angle}{(360 / 12) * 12};
    	\node[vertex,xshift=6cm,yshift=.5cm] (T11) at (\angle:1.5cm) {$T_{11}$};

    \draw[line width=.8pt] (T1) -- (T4);
    \draw[line width=.8pt] (T1) -- (T9);
    \draw[line width=.8pt] (T1) -- (T12);
    \draw[line width=.8pt] (T2) -- (T7);
    \draw[line width=.8pt] (T2) -- (T12);
    \draw[line width=.8pt] (T3) -- (T9);
    \draw[line width=.8pt] (T3) -- (T11);
    \draw[line width=.8pt] (T4) -- (T5);
    \draw[line width=.8pt] (T5) -- (T6);
    \draw[line width=.8pt] (T6) -- (T8);
    \draw[line width=.8pt] (T6) -- (T10);
    \draw[line width=.8pt] (T7) -- (T8);
    \draw[line width=.8pt] (T9) -- (T12);
    \draw[line width=.8pt] (T10) -- (T11);
    \end{tikzpicture}
    \qquad\qquad
      \begin{tikzpicture}[scale=.75, transform shape]
    \tikzstyle{vertex}=[circle,fill=black!30,minimum size=15pt,inner sep=0pt]
    \tikzstyle{vertexblack}=[circle,fill=black,minimum size=15pt,inner sep=0pt]
    \tikzstyle{vertexwhite}=[circle,fill=white,minimum size=15pt,draw=black,inner sep=0pt]

    \node[vertex,xshift=6cm,yshift=.5cm] (T12) at (-0.5,1) {$T_{9}$};
    \node[vertex,xshift=6cm,yshift=.5cm] (T9) at (.5,1) {$T_{12}$};
    
    \node[vertex,xshift=6cm,yshift=.5cm] (T7) at (-2,0) {$T_{5}$};
    \node[vertex,xshift=6cm,yshift=.5cm] (T8) at (-1,0) {$T_{6}$};
    \node[vertex,xshift=6cm,yshift=.5cm] (T6) at (0,0) {$T_{10}$};
    \node[vertex,xshift=6cm,yshift=.5cm] (T10) at (1,0) {$T_{8}$};
    \node[vertex,xshift=6cm,yshift=.5cm] (T11) at (2,0) {$T_{7}$};
    
     \node[vertex,xshift=6cm,yshift=.5cm] (T5) at (0,-1) {$T_{11}$};
    
    \draw[line width=.8pt] (T5) -- (T6);
    \draw[line width=.8pt] (T6) -- (T8);
    \draw[line width=.8pt] (T6) -- (T10);
    \draw[line width=.8pt] (T7) -- (T8);
    \draw[line width=.8pt] (T9) -- (T12);
    \draw[line width=.8pt] (T10) -- (T11);
   
    \end{tikzpicture}
\end{center}
\item[\tt >] \tt{MDSintersno(X2); }
\out{ [-1, -1, -1, -1, -2, -2, -2, -2, -2, -2, -2, -2]}
\end{enumerate}
\endgroup

\noindent
The plot of the intersection graph are provided 
by \texttt{MDSpackage} either as \texttt{Maple}-objects
or in \LaTeX-format.
\end{example}

\begin{example}
We consider a variety with torus action of 
complexity one. 
The Cox ring $R$ of any such variety $Y$ can be 
simply encoded by a pair $(A,P)$ of matrices, 
see e.g.~\cite{ArDeHaLa} for details.
Here is how to enter it in~\texttt{MDSpackage}:
\\

\begingroup
\footnotesize
\begin{enumerate}[leftmargin=3em]
\item[\tt >] \tt{P := cols2matrix([
[-2,-2,-1,-1], [1,0,0,0], [1,0,1,0], [0,1,0,1], [0,1,0,0]
])}
\out{P := \left[
\begin{array}{rrrrr}
-2 & 1 & 1 & 0 & 0\\
-2 & 0 & 0 & 1 & 1\\
-1 & 0 & 1 & 0 & 0\\
-1 & 0 & 0 & 1 & 0
\end{array}
\right]}
\item[\tt >] \tt{A := [[1,0],[0,1],[-1,-1]];}
\out{A := [[1,0],[0,1],[-1,-1]]}
\item[\tt >] \tt{R := createGR(P, A);}
\out{R := GR(5, 1, [1, []])}
\item[\tt >] \tt{GRdata(R);}
\out{
\Bigl[[T[4] T[5] + T[2] T[3] + T[1]^2 ], [T[1], T[2], T[3], T[4], T[5]],
\\
\bigl[AGH([5, []], [1, []]), \begin{bmatrix}
                        1 & 1 & 1 & 1 & 1
                        \end{bmatrix}\bigr],
                        \ldots
    \bigr]
}
\end{enumerate}
\endgroup

\noindent
So far we defined the prospective Cox ring $R$.
To specify a variety $Y$ with Cox ring $R$,
we have to fix a bunch $\Phi$.
However, the unit component ${\rm Aut}(Y)^0$ of the 
(linear algebraic) automorphism group only depends
on the Cox ring
and we can compute its roots as follows:\\

\begingroup
\footnotesize
\begin{enumerate}[leftmargin=3em]
\item[\tt >] \tt{MDSautroots(R);}
\out{\{[1, -1], [1, 1], [-1, -1], [-1, 1], [0, -1], [0, 1], [1, 0], [-1, 0]\}}
\end{enumerate}
\endgroup

\noindent
We have obtained the root 
system~$B_2$ of the corresponding orthogonal group,
as expected for the smooth quadric
$Y$ in~$\PP_5$.
\begin{center}
\tiny
\begin{tikzpicture}[scale=.45]
\fill[color=black!25] (-1,-1) -- (1,-1) -- (1,1) -- (-1,1) -- cycle;
\draw[->] (-1.5,0) -- (1.5,0);
\draw[->] (0,-1.5) -- (0,1.5);

\foreach \x/\y in { -1/-1, 1/-1, 1/1, -1/1, -1/0, 1/0, 0/1, 0/-1 }{
  \fill (\x,\y) circle(2.5pt);
  \draw[thick] (0,0) -- (\x,\y);
}
\end{tikzpicture}
\end{center}
\end{example}

\begin{remark}
As a support of \texttt{MDSpackage} we build up the 
\emph{Cox ring database}~\cite{coxringdb}.
This database stores known Cox rings of varieties. 
Among others, it provides at the moment 
the Cox rings of all non-toric
\begin{itemize}
\item
smooth rational surfaces up to Picard number five, 
\item
smooth rational surfaces of Picard number six 
admitting only trivial $\KK^*$-actions, 
\item
Gorenstein del Pezzo 
surfaces of Picard number one, 
\item
cubic surfaces with at most rational 
double points,
\item
Gorenstein del Pezzo $\KK^{*}$-surfaces,
\item
smooth Fano threefolds of Picard number 
at most two,
\item
terminal Fano threefolds of Picard number one 
with a two-torus action.
\end{itemize}
We use the representation in terms of bunched rings 
explained in Section~\ref{sec:working}:
for a given Mori dream space $X$, the database stores 
the tuple $(r,\{g_1,\ldots,g_s\},K,Q,\Phi)$ where
the divisor class group and the Cox ring of $X$ are 
given by
$$ 
\Cl(X) \ = \ K, 
\qquad\qquad
\mathcal{R}(X) \ = \ \KT{r}/\<g_1,\ldots,g_s\>,
$$
the entry $Q$ is the degre matrix and $\Phi$ is
the bunch of cones fixing $X$ --- in case of a 
projective $X$ we just replace $\Phi$ with an 
ample class $w \in K$.
Here are some basic features of the database:
\begin{itemize}
\item
If available, further information is stored
such as basic geometric properties (dimension, 
Picard number, type of singularities, etc.), 
reference to author and original literature,
\item 
The web interface~\cite{coxringdb} allows searching 
the Cox ring database according to key words, 
authors and geometric properties,
\item 
The stored Cox rings and varieties can be directly 
exported in the necessary syntax for \texttt{MDSpackage} 
or \LaTeX.
\end{itemize}
For instance, the $40$ Fano threefolds 
of~\cite{BeHaHuNu} are found by entering 
\texttt{Fano AND terminal AND 
Picard number one AND dimension three AND Q-factorial} 
in the search field \texttt{description}.
Searching the \texttt{id} field for $96$, $97$, $6$ and~$98$
yields the Cox rings discussed in the introduction and the 
previous three examples.
\end{remark}

\section{Application: Resolutions of a quotient singularity}
\label{sec:app}

We use \texttt{MDSpackage} to study a symplectic quotient 
singularity discussed by Bellamy/Schedler~\cite{BellamySchedler} 
and by Donten-Bury/Wi\'{s}niewski~\cite{JarekMaria}. 
More precisely, we consider the subgroup 
$G \subseteq \mathrm{Gl}_4(\mathbb{C})$
generated by the matrices 
$$ 
\left[
\mbox{\footnotesize $
\begin{array}{rrrr}
1 & 0 & 0 & 0
\\
0 & -1 & 0 & 0
\\
0 & 0 & 1 & 0
\\
0 & 0 & 0 & -1
\end{array}
$
}
\right],
\qquad
\left[
\mbox{\footnotesize $
\begin{array}{rrrr}
0 & \I & 0 & 0
\\
-\I & 0 & 0 & 0
\\
0 & 0 & 0 & -\I
\\
0 & 0 & \I & 0
\end{array}$
}
\right],
\qquad
\left[
\mbox{\footnotesize $
\begin{array}{rrrr}
0 & 1 & 0 & 0
\\
1 & 0 & 0 & 0
\\
0 & 0 & 0 & 1
\\
0 & 0 & 1 & 0
\end{array}$
}
\right],
$$
$$
\left[
\mbox{\footnotesize $
\begin{array}{rrrr}
0 & 0 & 0 & 1
\\
0 & 0 & -1 & 0
\\
0 & -1 & 0 & 0
\\
1 & 0 & 0 & 0
\end{array}$
}
\right],
\qquad
\left[
\mbox{\footnotesize $
\begin{array}{rrrr}
0 & 0 & 0 & \I
\\
0 & 0 & -\I & 0
\\
0 & \I & 0 & 0
\\
-\I & 0 & 0 & 0
\end{array}$
}
\right].
$$ 

\noindent
The aim is to study resolutions of singularities 
of the quotient space $\mathbb{C}^4/G$.
In Theorem~\ref{thm:81resol}, we will determine 
for one of them the Cox ring and show that all 
of its flops are as well resolutions of singularities. 
This continues~\cite{BellamySchedler,JarekMaria}
and also retrieves the 81 symplectic 
resolutions of $\CC^4/G$ discussed there.

Observe that the commutator subgroup $[G,G] \subseteq G$ 
is generated by the negative unit matrix $-E_4$ and hence is
of order two.
Moreover, the abelianization $G' := G/[G,G]$ is isomorphic to 
$(\mathbb{Z}/2\mathbb{Z})^4$ and generated by the classes 
of the last four matrices.
A first step is to determine the Cox ring of the quotient 
space $\mathbb{C}^4/G$.

\begin{proposition}\label{prop1}
The quotient $\mathbb{C}^4/[G,G] \to \mathbb{C}^4/G$ 
by the induced action of $G'$ is the characteristic 
space over $\mathbb{C}^4/G$. 
The Cox ring $\mathcal{R}(\mathbb{C}^4/G) \cong 
\mathbb{C}[T_1,\ldots,T_{10}]^{[G,G]}$ is isomorphic
to $\mathbb{C}[T_1,\ldots,T_{10}] / I$, where 
the ideal $I$ is generated by
\begingroup
\footnotesize
\begin{gather*}
T_{5}T_{6}+T_{8}T_{9}-T_{7}T_{10}, \quad
T_{4}T_{6}+T_{2}T_{7}+T_{1}T_{8}, \quad
T_{3}T_{6}-T_{2}T_{9}-T_{1}T_{10},
\\
T_{2}T_{6}+T_{4}T_{7}+T_{3}T_{9}, \quad
T_{1}T_{6}-T_{4}T_{8}-T_{3}T_{10}, \quad
T_{5}^2-T_{7}^2+T_{9}^2+T_{2}^2,
\\
T_{4}T_{5}+T_{1}T_{9}+T_{2}T_{10}, \quad
T_{3}T_{5}-T_{1}T_{7}-T_{2}T_{8}, \quad
T_{2}T_{5}+T_{3}T_{8}+T_{4}T_{10},
\\
T_{1}T_{5}-T_{3}T_{7}-T_{4}T_{9}, \quad
T_{4}^2-T_{6}^2-T_{9}^2+T_{10}^2, \quad
T_{3}T_{4}+T_{7}T_{9}-T_{8}T_{10},
\\
T_{2}T_{4}+T_{6}T_{7}-T_{5}T_{10}, \quad
T_{1}T_{4}-T_{6}T_{8}+T_{5}T_{9}, \quad
T_{3}^2+T_{6}^2-T_{7}^2+T_{8}^2,
\\
T_{2}T_{3}+T_{5}T_{8}-T_{6}T_{9}, \quad
T_{1}T_{3}-T_{5}T_{7}+T_{6}T_{10},\quad 
T_{2}^2-T_{6}^2-T_{8}^2+T_{10}^2,
\\
T_{1}T_{2}+T_{7}T_{8}-T_{9}T_{10}, \qquad
T_{1}^2+T_{6}^2-T_{7}^2+T_{9}^2
\end{gather*}
\endgroup

\noindent
and the grading of $\mathbb{C}[T_1,\ldots,T_{10}] / I$ by 
$\mathrm{Cl}(\mathbb{C}^4/G) \cong \mathbb{X}(G') 
\cong (\mathbb{Z}/2\mathbb{Z})^4$
is given by assigning to the class of $T_i$ the 
$i$-th column of the following matrix over 
$\mathbb{Z}/2\mathbb{Z}$:
$$ 
Q
\ := \ 
\left[
\mbox{ \footnotesize $
\begin{array}{rrrrrrrrrr}
0 & 0 & 0 & 1 & 1 & 0 & 1 & 1 & 0 & 0
\\
0 & 0 & 0 & 1 & 0 & 1 & 0 & 0 & 1 & 1
\\
1 & 0 & 1 & 0 & 0 & 0 & 0 & 1 & 1 & 0
\\
0 & 1 & 1 & 0 & 0 & 0 & 1 & 0 & 0 & 1
\end{array}$
}
\right].
$$
The variables $T_i$ define pairwise non-associated 
$\mathbb{X}(G')$-prime elements in
$\mathcal{R}(\mathbb{C}^4/G)$.
The above presentation of the Cox ring gives rise 
to a commutative diagram
$$ 
\xymatrix{
{\mathbb{C}^4/[G,G]}
\ar[rr]
\ar[d]_{/G'}
&&
{\mathbb{C}^{10}}
\ar[d]^{/G'}
\\
{\mathbb{C}^4/G}
\ar[rr]
&&
Z
}
$$
where the horizontal arrows are closed embeddings. 
The ambient affine toric variety $Z = \mathbb{C}^{10}/G'$
of $\mathbb{C}^4/G$ is of dimension ten and arises 
from the cone $\sigma \subseteq \mathbb{Q}^{10}$ 
generated by the vectors
\begingroup \small
\begin{gather*} 
v_1 \,:=\, e_1, \quad 
v_2 \,:=\, e_2, \quad 
v_3 \,:=\, e_3, \quad 
v_4 \,:=\, e_4, \quad 
v_5 \,:=\, e_5, 
\\
v_6 \,:=\, e_1+e_2+e_5+2e_6, \quad
v_7 \,:=\, e_7 \quad
v_8 \,:=\, e_4+e_5+e_7+2e_8, \quad
\\
v_9 \,:=\, e_1+e_3+e_4+e_5+e_7+2e_8, \quad 
v_{10} \,:=\, e_1+e_2+e_7+2e_{10}. 
\end{gather*}
 \endgroup

\noindent
 There are precisely five toric orbits of dimension six in $Z$ that contain 
singularities of $\mathbb{C}^4/G$.
Any further singularities of $\mathbb{C}^4/G$ are contained in the 
closures of these orbits.
The faces of $\sigma$ corresponding to these five toric orbits are:
$$
\mathrm{cone}(v_1,v_2,v_5,v_6),
\quad
\mathrm{cone}(v_4,v_5,v_7,v_8),
\quad
\mathrm{cone}(v_1,v_3,v_8,v_9),
$$
$$
\mathrm{cone}(v_4,v_6,v_9,v_{10}),
\qquad
\mathrm{cone}(v_2,v_3,v_7,v_{10}).
$$
\end{proposition}

\begin{proof}
The first sentence is a direct consequence of 
a more general statement due to Arzhantsev and 
Ga$\breve{\text{\i}}$fullin~\cite{AG10}.
We verify the claimed presentation of the Cox 
ring 
$$
\mathcal{R}(\mathbb{C}^4/G) 
\ \cong \ 
\mathbb{C}[S_1,\ldots,S_4]^{[G,G]}.
$$
The algebra of $[G,G]$-invariant polynomials 
is generated by the quadratic monomials 
of $\mathbb{C}[S_1,\ldots,S_4]$.
Moreover, $G'$ acts on the ten-dimensional 
$\mathbb{C}$-vector subspace generated by 
these monomials and the following 
$G'$-homoge\-neous polynomials form a basis 
of this vector subspace: 
\begingroup \small
\begin{gather*}
S_1S_2 + S_3S_4, \quad
S_1S_2 - S_3S_4, \quad 
S_1S_3 + S_2S_4, 
\\
S_1S_3 - S_2S_4, \quad
S_1S_4 + S_2S_3, \quad 
S_1S_4 - S_2S_3,
\\
\frac{1}{2}\left(S_1^2+ S_2^2 + S_3^2 + S_4^2\right), \qquad  
\frac{1}{2}\left(-S_1^2 - S_2^2 + S_3^2 + S_4^2\right),
\\
\frac{1}{2}\left(-S_1^2 + S_2^2 - S_3^2 + S_4^2\right), \qquad
\frac{1}{2}\left(S_1^2  - S_2^2 - S_3^2 + S_4^2\right).
\end{gather*}
\endgroup
According to the order of listing, the weights 
of these polynomials occur as columns of the 
matrix~$Q$ of the assertion.
Thus, we found ten $\mathbb{X}(G')$-homogeneous 
prime generators for  
$\mathbb{C}[S_1,\ldots,S_4]^{[G,G]}$.

The desired presentation of the Cox ring is now 
obtained by sending $T_i$ to the $i$-th generator 
and the ideal $I$ is computed as the ideal of 
algebraic relations between the above generators.
The remaining claims are direct applications 
of~\cite[Sections~3.2.5 and~3.3.1]{ArDeHaLa}.
\end{proof}

The second step is to compute the Cox ring $\mathcal{R}(X)$
for a certain resolution $X \to \mathbb{C}^4/G$ 
of singularities.
This is done using the technique of toric ambient 
modifications introduced in~\cite{Ha2}: 
we just perform an evident partial toric resolution 
of the ambient affine toric variety $Z$ 
and show that the proper transform $X$ of $\mathbb{C}^4/G$ 
is as wanted. 
The 81 symplectic resolutions then correspond to 
the full dimensional chambers of the Mori chamber 
decomposition of the cone of movable divisor classes 
of~$X$.

\begin{theorem}
\label{thm:81resol}
Consider the polynomial ring 
$\mathbb{C}[T_1,\ldots,T_{15}]$
with the $\mathbb{Z}^5$-grading assigning 
to the variables the columns of the following 
matrix
$$
\left[
\mbox{\footnotesize $
\begin{array}{rrrrrrrrrrrrrrr}
1 & 1 & 0 & 0 & 1 & 1 & 0 & 0 & 0 & 0 & -2 & 0 & 0 & 0 & 0
\\
0 & 0 & 0 & 1 & 1 & 0 & 1 & 1 & 0 & 0 & 0 & -2 & 0 & 0 & 0
\\
1 & 0 & 0 & 0 & 0 & 1 & -1 & 0 & 1 & 0 & -1 & 1 & -1 & -1 & 1
\\
1 & 1 & 1 & -1 & 0 & 0 & 0 & 0 & 0 & 0 & -1 & 1 & -1 & 1 & -1
\\
0 & 1 & 0 & 0 & 0 & 1 & 0 & -1 & 0 & 1 & -1 & 1 & 1 & -1 & -1
\end{array}$}
\right]
$$
and the ideal 
$I \subseteq \mathbb{C}[T_1,\ldots,T_{15}]$
generated by the following $\mathbb{Z}^5$-homogeneous
polynomials
\begingroup
\allowdisplaybreaks
\footnotesize
\begin{gather*}
T_{3}T_{6}-T_{2}T_{9}-T_{1}T_{10}, \quad 
T_{3}T_{5}-T_{1}T_{7}-T_{2}T_{8}, \quad 
T_{3}T_{4}+T_{7}T_{9}-T_{8}T_{10}, 
\\
T_{2}T_{4}+T_{6}T_{7}-T_{5}T_{10}, \qquad 
T_{1}T_{4}-T_{6}T_{8}+T_{5}T_{9}, 
\\
T_{1}T_{8}T_{13}+T_{4}T_{6}T_{14}+T_{2}T_{7}T_{15}, \quad 
T_{5}T_{8}T_{12}-T_{6}T_{9}T_{14}+T_{2}T_{3}T_{15}, 
\\
T_{5}T_{7}T_{12}-T_{1}T_{3}T_{13}-T_{6}T_{10}T_{14}, \quad 
T_{5}^2T_{12}-T_{1}^2T_{13}-T_{6}^2T_{14}+T_{2}^2T_{15}, \quad 
\\
T_{4}T_{5}T_{12}+T_{1}T_{9}T_{13}+T_{2}T_{10}T_{15}, \quad 
T_{6}^2T_{11}-T_{4}^2T_{12}+T_{9}^2T_{13}-T_{10}^2T_{15}, 
\\
T_{5}T_{6}T_{11}+T_{8}T_{9}T_{13}-T_{7}T_{10}T_{15}, \quad
T_{2}T_{6}T_{11}+T_{4}T_{7}T_{12}+T_{3}T_{9}T_{13}, 
\\
T_{1}T_{6}T_{11}-T_{4}T_{8}T_{12}-T_{3}T_{10}T_{15}, \quad 
T_{5}^2T_{11}+T_{8}^2T_{13}+T_{4}^2T_{14}-T_{7}^2T_{15}, 
\\
T_{2}T_{5}T_{11}+T_{3}T_{8}T_{13}+T_{4}T_{10}T_{14}, \quad 
T_{1}T_{5}T_{11}-T_{4}T_{9}T_{14}-T_{3}T_{7}T_{15}, 
\\
T_{2}^2T_{11}-T_{7}^2T_{12}+T_{3}^2T_{13}+T_{10}^2T_{14}, \quad 
T_{1}T_{2}T_{11}+T_{7}T_{8}T_{12}-T_{9}T_{10}T_{14}, 
\\
T_{1}^2T_{11}-T_{8}^2T_{12}+T_{9}^2T_{14}-T_{3}^2T_{15}.
\end{gather*}
\endgroup
Then the ideal $I$ is homogeneous w.r.t.~the 
$\mathbb{Z}^5$-grading and thus, the factor ring 
$R := \mathbb{C}[T_1,\ldots,T_{15}] / I$
inherits a $\mathbb{Z}^5$-grading.
Moreover, the following statements hold.
\begin{enumerate}
\item
The ring $R$ is factorial, the variables 
$T_1,\ldots,T_{15}$ define pairwise non-associated 
prime elements in $R$ and any $14$ of the degrees 
of the $15$ variables generate $\mathbb{Z}^5$ as 
a group.
\item
The moving cone $\mathrm{Mov}(R)$ of the 
$\mathbb{Z}^5$-graded ring $R$ is of dimension five
and contains precisely 81 five-dimensional 
GIT-chambers.
Each of these chambers defines a smooth variety 
$X_i$ having divisor class group $\mathrm{Cl}(X_i)
\cong \mathbb{Z}^5$ and Cox ring 
$\mathcal{R}(X_i) = R$.
\item
The subring $R_0 \subseteq R$ of elements of degree
zero is isomorphic to the coordinate ring of 
$\CC^4/G$ and thus we have canonical morphisms
$X_i \to \CC^4/G$, each of which is a symplectic
resolution of singularities.
\end{enumerate}
\end{theorem}

\begin{proof}
We start with $X_0 := \CC^4/G$ and the presentation of its 
Cox ring $R_0 := \mathcal{R}(X_0)$
given in Proposition~\ref{prop1}.
We enter the relations of $R_0$ given there as a list \texttt{RL0} 
in \texttt{MDSpackage} and as well the grading matrix 
which we call now \texttt{Q0mat}.
Using these data, we then define the degree map $Q_0\colon \ZZ^{10} \to (\ZZ/2\ZZ)^4$, 
the Cox ring $R_0$ and $X_0$:\\

\begingroup
\footnotesize
\begin{enumerate}[leftmargin=3em]
\item[\tt >] \tt{Q0 := createAGH(createAG(10),createAG(0,[2,2,2,2]),Q0mat);}
\out{Q0 := AGH([10, []], [0, [2, 2, 2, 2]])}
\item[\tt >] \tt{R0 := createGR(RL0, vars(10), [Q0], 'Singular');}
\out{R0 := GR(10, 20, [0, [2, 2, 2, 2]])}
\item[\tt >] \tt{X0 := createMDS(R0,[0]);}
\out{X0 := MDS(10, 20, 4, [0, [2, 2, 2, 2]])}
\end{enumerate}
\endgroup

We resolve the singularities of $X_0$
via a suitable toric modification
of the ambient affine toric variety $Z$.
By Proposition~\ref{prop1}, there are five singular 
toric orbits, the closures of which form
the singular locus of $Z$.
The first evident step towards any toric resolution
is to resolve the corresponding cones. 
This is achieved in each of the five cases by a 
single barycentric subdivision, i.e., we insert 
five new rays $\QQ_{\geq 0}\cdot v_1,\ldots,\QQ_{\geq 0}\cdot v_5$
in total where the list of $v_i$ is\\

\begingroup
\footnotesize
\begin{enumerate}[leftmargin=3em]
\item[\tt >] \tt{VL := [[1,1,0,0,1,1,0,0,0,0], [0,0,0,1,1,0,1,1,0,0], [1,0,1,1,1,0,1,1,1,0],}\\
\tt{[1,1,1,1,1,1,1,0,1,1], [0,1,1,0,0,0,1,0,0,1]];}
\out{VL := [[1, 1, 0, 0, 1, 1, 0, 0, 0, 0], [0, 0, 0, 1, 1, 0, 1, 1, 0, 0], [1, 0, 1, 1, 1, 0, 1, 1, 1, 0],\\
[1, 1, 1, 1, 1, 1, 1, 0, 1, 1], [0, 1, 1, 0, 0, 0, 1, 0, 0, 1]]}
\end{enumerate}
\endgroup

\noindent
The Cox ring $R$ of the proper transform $X$ of $X_0$ under 
such a partial ambient resolution can be computed 
via  the command \texttt{MDSmodify} which 
implements~\cite[Algorithm~3.6]{HaKeLa}:\\

\begingroup
\footnotesize
\begin{enumerate}[leftmargin=3em]
\item[\tt >] \tt{X := MDSmodify(X0, VL, 'Singular', 'verify');}\\[1ex]
Verification successful: all variables are prime.\\
\out{X := MDS(15, 20, 4, [5, []])}
\item[\tt >] \tt{R := MDSdata(X)[1];}
\out{R := GR(15, 20, [5, []])}
\end{enumerate}
\endgroup

\noindent
The algorithm thus returns the $\mathbb{Z}^5$-graded ring \texttt{R}
as the Cox ring $R:=R_2$ of the assertion and, additionally, verifies the 
statements of the first assertion. 
For the second assertion, we compute the GIT-fan
using \texttt{MDSpackage}'s 
implementation of~\cite[Algorithm~8]{Ke}:\\

\begingroup
\footnotesize
\begin{enumerate}[leftmargin=3em]
\item[\tt >] \tt{F := GRgitfan(R, 'fan');}
\out{F := FAN(5, 0, [0, 0, 0, 0, 207])}
\item[\tt >] \tt{Mov := MDSmov(X);}
\out{Mov := CONE(5, 5, 0, 5, 5)}
\end{enumerate}
\endgroup

\noindent
The command \texttt{U := select(c -> contains(Mov,c), maximal(F))}
retrieves the 81 full-dimensional chambers 
inside the moving cone as claimed.
We test smoothness of the varieties $X_{i}$
associated to the chambers \texttt{U[i]}
using an \texttt{MDSpackage} function 
based on the criterion~\cite[Cor.~3.3.1.12]{ArDeHaLa};
compare~\cite{JarekMaria} for a related 
approach:\\

\begingroup
\footnotesize
\begin{enumerate}[leftmargin=3em]
\item[\tt >] \tt{w1 := relint(U[1])}
\out{w1 := [6, 1, 3, 3, 3]}
\item[\tt >] \tt{X1 := createMDS(R, w1);}
\out{X1 := MDS(15, 20, 4, [5, []])}
\item[\tt >] \tt{MDSissmooth(X1, 'magma'); }
\out{true}
\end{enumerate}
\endgroup

\noindent
All test results are positive. The computation
is sped up a bit by calling \texttt{magma} under
the surface; it takes some hours on a personal
computer.
Note that $X$ is among the $X_{i}$.
The third assertion is clear by the way we obtained 
the common Cox ring $R$ of the~$X_{i}$.
\end{proof}

\begin{remark}
Having proved Theorem~\ref{thm:81resol}, 
a discussion of the geometry of the resolutions 
is possible via the methods of~\texttt{MDSpackage}.
\end{remark}

\begin{acknowledgement}
We are grateful to Jaros\l aw Wi\'sniewski for 
informing us about the current state of research
concerning the 81 symplectic resolutions just discussed
and encouraging us to work out Theorem~\ref{thm:81resol}.
\end{acknowledgement}

\bibliographystyle{abbrv}

\end{document}